\documentclass[12pt]{amsart}
\usepackage{amsmath,amssymb,color,cite,verbatim}
\usepackage{etoolbox}
\apptocmd{\sloppy}{\hbadness 10000\relax}{}{}

\numberwithin{equation}{section}

\newtheorem{thm}[equation]{Theorem}

\newtheorem{lemma}[equation]{Lemma}
\newtheorem{cor}[equation]{Corollary}

\theoremstyle{definition}
\newtheorem{rmk}[equation]{Remark}

\newcommand{\F}{\mathbb{F}}
\newcommand{\bP}{\mathbb{P}}
\newcommand{\C}{\mathbb{C}}

\newcommand{\g}{\mathfrak g}

\DeclareMathOperator{\charp}{char}
\DeclareMathOperator{\lcm}{lcm}

\newcommand{\mybar}[1]{#1\llap{$\overline{\phantom{\rm#1}}$}}
\newcommand{\abs}[1]{\lvert #1 \rvert}

\usepackage[colorlinks,pagebackref,pdftex, bookmarks=false]{hyperref}

\begin{document}

\title{The relative Riemann--Hurwitz formula}

\author{Zhiguo Ding}
\address{
  Hunan Institute of Traffic Engineering,
  Heng\-yang, Hunan 421001 China
}
\email{ding8191@qq.com}

\author{Michael E. Zieve}
\address{
  Department of Mathematics,
  University of Michigan,
  530 Church Street,
  Ann Arbor, MI 48109-1043 USA
}
\email{zieve@umich.edu}
\urladdr{http://www.math.lsa.umich.edu/$\sim$zieve/}

\date{\today}

\begin{abstract}
For any nonconstant $f, g \in \C(x)$ such that the numerator $H(x,y)$ of $f(x) - g(y)$ 
is irreducible, we compute the genus of the normalization of the curve $H(x,y) = 0$. 
We also prove an analogous formula in arbitrary characteristic when $f$ and $g$ have 
no common wildly ramified branch points, and generalize to (possibly reducible) fiber 
products of nonconstant morphisms of curves $f \colon A \to D$ and $g \colon B \to D$.
\end{abstract}

\thanks{
The second author thanks the National Science Foundation for support under grant DMS-1601844.}

\maketitle


\section{Introduction} 

Mathematicians studying several different topics have been led to investigate the genus of algebraic 
curves of the form $f(x)=g(y)$. For instance, such considerations arise in algebraic topology 
\cite{EKS}, complex dynamics \cite{Ye}, value distribution of meromorphic functions \cite{B,KA,KAL,Wang}, 
Diophantine equations \cite{AZ1,AZ0,BST,BT,DLS}, arithmetic dynamics \cite{CJS,CHZ}, Pell equations 
in polynomials \cite{AZ1,Ng}, functional equations in rational functions \cite{DW,Ritt,Rittrat,Z}, and coding 
theory \cite{Rod}. This has led many authors to compute the genus of some specific instance of such 
curves.

Besides specific instances, mathematicians have produced two general types of formulas for this genus.  
Ritt \cite{Rittrat} essentially determined the genus of irreducible curves of the form $f(x)=g(y)$ when $f(x)$ and $g(x)$ 
are nonconstant complex rational functions. His method uses the Riemann--Hurwitz formula for the projection 
of this curve onto the $x$-axis, and relies on computing the ramification in this projection. By 
examining singularities and using the Pl\"ucker formula, Kang \cite{Kang} proved a genus formula for irreducible curves of the form $f(x)=g(y)$ where $f(x)$ and $g(x)$ are nonconstant polynomials
over an algebraically closed field of characteristic zero. The genus formulas 
obtained by Ritt and Kang have different forms, and are not obviously equivalent. 

In this paper we prove new results generalizing these known formulas in several ways:
\begin{itemize}
\item We do not assume that $f(x)=g(y)$ is irreducible, but instead obtain a sum over its irreducible 
components.
\item Our formulas are valid in positive characteristic, under some hypotheses about wild ramification.
\item We do not require $f$ and $g$ to be polynomials (as did Kang) or rational functions (as did Ritt), but instead allow them to be arbitrary curve morphisms with a common target.
\end{itemize}

Our approach uses ramification, and in particular the computation of points in fibered products from 
\cite{DZ}, but it differs from Ritt's approach in several ways. In the end we obtain two formulas: 
one generalizing Ritt's which is valid when at least one of $f$ and $g$ is tamely ramified, and one generalizing 
Kang's in the more general situation where $f$ and $g$ do not have a common wildly ramified branch point.

Our main result uses the following standard notation and terminology, where $f \colon A \to D$ is a nonconstant morphism between smooth projective irreducible curves over an algebraically closed field $k$ of characteristic $p\ge 0$:
\begin{itemize}
\item $A(k)$ is the set of points on $A$ with coordinates in $k$
\item for any $P\in A(k)$, we write $e_f(P)$ for the ramification index of $P$ under $f$
\item we say that $P\in A(k)$ is tamely ramified (resp., wildly ramified) under $f$ if $p\nmid e_f(P)$ (resp., $p\mid e_f(P)$), and that $f$ is tamely ramified if every point in $A(k)$ is tamely ramified under $f$; in particular, if $p=0$ then $f$ is tamely ramified
\item a branch point of $f$ is a point in $D(k)$ of the form $f(P)$ where $P\in A(k)$ satisfies $e_f(P)>1$
\item  a branch point of $f$ is wildly ramified if it has a wildly ramified $f$-preimage
\end{itemize}

\begin{thm}\label{main}
Let $f \colon A \to D$ and 
$g \colon B \to D$ be nonconstant morphisms of smooth projective irreducible curves over an algebraically closed field $k$. Write $m$ 
and $n$ for the degrees of $f$ and $g$, and write $\g_A,\g_B,\g_D$ for the genera of $A,B,D$. Let 
$C_1,\dots,C_r$ be the normalizations of the irreducible components of the fiber product of $f$ and 
$g$, and write $\g_i$ for the genus of $C_i$.
\renewcommand{\theenumi}{\ref{main}.\arabic{enumi}}
\renewcommand{\labelenumi}{(\thethm.\arabic{enumi})}
\begin{enumerate}
\item \label{maintame} If $f$ is tamely ramified then
\[
\sum_{i=1}^r (2\g_i-2) = m(2\g_B-2) + 
\sum_{R\in D(k)}\,\sum_{\substack{P\in f^{-1}(R) \\ Q\in g^{-1}(R)}}\bigl(e_f(P)-\gcd\bigl(e_f(P),e_g(Q)\bigr)\bigr).
\]
\item \label{mainwild} If $f$ and $g$ have no common wildly ramified branch points then
\begin{align*}
& \sum_{i=1}^r (2\g_i-2) = m(2\g_B-2) + n(2\g_A-2) - mn(2\g_D-2) \,-\\
& \quad\,\,\,\,\! \sum_{R\in D(k)}\,\sum_{\substack{P\in f^{-1}(R) \\ Q\in g^{-1}(R)}}\bigl((e_f(P)-1) \cdot (e_g(Q)-1)+\gcd\bigl(e_f(P),e_g(Q)\bigr)-1\bigr).
\end{align*}
\end{enumerate}
\end{thm}
\renewcommand{\theenumi}{\arabic{enumi}}
\renewcommand{\labelenumi}{(\arabic{enumi})}

\begin{rmk}
In case $g$ is the identity map on $D$, the formula in \eqref{maintame} is the classical Riemann--Hurwitz 
formula for tamely ramified morphisms. For more general $g$, the formulas in \eqref{maintame} and \eqref{mainwild}
may be viewed as ``relative'' Riemann--Hurwitz formulas.
\end{rmk}

\begin{rmk}
The summations in the above result only have finitely many nonzero summands, since the summation in the right side of \eqref{maintame} is zero whenever $e_f(P)=1$, and the summation in the right side of \eqref{mainwild} is zero when either $e_f(P)=1$ or $e_g(Q)=1$.
\end{rmk}

In order to make it easy to apply Theorem~\ref{main}, we now state some special cases of this result. We 
begin with the case of rational functions.

\begin{cor}\label{introrat}
Let $k$ be an algebraically closed field of characteristic $p\ge 0$, and let $f,g\in k(x)\setminus k$ have 
degrees $m$ and $n$, respectively. Let $C_1,\dots,C_r$ be the normalizations of the irreducible components 
of $f(x)=g(y)$, and write $\g_i$ for the genus of $C_i$.
\renewcommand{\theenumi}{\ref{introrat}.\arabic{enumi}}
\renewcommand{\labelenumi}{(\thethm.\arabic{enumi})}
\begin{enumerate}
\item \label{rattame} Suppose that $p\nmid e_f(P)$ for each $P\in\bP^1(k)$.
 Let $\Gamma$ be the 
set of pairs $(P,Q) \in \bP^1(k) \times \bP^1(k)$ such that $f(P) = g(Q)$ and $e_f(P)>1$. Then $\Gamma$ 
is finite, say $\Gamma = \{ (P_j,Q_j) \colon 1 \le j \le s \}$ with $s = \abs{\Gamma}$;
writing $m_j := e_f(P_j)$ and $n_j := e_g(Q_j)$, we have 
\[
\sum_{i=1}^r (2\g_i-2) = -2m + \sum_{j=1}^s\Bigl(m_j - \gcd(m_j,n_j)\Bigr).
\]
\item \label{ratwild} Let $\Lambda$ be the set of pairs $(P,Q) \in \bP^1(k) \times \bP^1(k)$ such that 
$f(P) = g(Q)$ and $e_f(P)>1$ and $e_g(Q)>1$. Suppose $\Lambda$ does not contain a pair of integers divisible 
by $p$. Then $\Lambda$ is finite, say $\Lambda = \{ (P_j,Q_j) \colon 1 \le j \le t \}$ with $t = \abs{\Lambda}$;
writing $m_j := e_f(P_j)$ and $n_j := e_g(Q_j)$, we have
\begin{align*}
\sum_{i=1}^r (2\g_i-2) &= -2 + 2(m-1)(n-1)\\
&\qquad - \sum_{j=1}^t\Bigl((m_j-1)(n_j-1) + \gcd(m_j,n_j)-1\Bigr).
\end{align*}
\end{enumerate}
\end{cor}
\renewcommand{\theenumi}{\arabic{enumi}}
\renewcommand{\labelenumi}{(\arabic{enumi})}

\begin{rmk}
For the benefit of those whose background is analysis rather than algebraic geometry, we note the following alternate description of the $C_i$'s in case $k=\C$: if we write the numerator of the bivariate rational function $f(x)-g(y)$ as the product $\prod_{i=1}^r H_i(x,y)$ of irreducible polynomials $H_i(x,y)\in\C[x,y]$, then $C_i$ is the unique compact Riemann surface (up to conformal equivalence) for which there exist a finite subset $S_i$ of $C_i$, and a finite subset $T_i$ of the set $U_i$ of zeroes of $H_i(x,y)$ in $\C^2$, such that $C_i\setminus S_i$ is biholomorphic to $U_i\setminus T_i$.
\end{rmk}

We now state our results for polynomials, where we incorporate the ramification 
over $\infty$ into the formulas.

\begin{cor}\label{intropol}
Let $k$ be an algebraically closed field of characteristic $p\ge 0$, and let $f,g\in k[x]\setminus k$ have degrees 
$m$ and $n$, respectively. Let $C_1,\dots,C_r$ be the normalizations of the irreducible components of $f(x)=g(y)$, 
and write $\g_i$ for the genus of $C_i$.
\renewcommand{\theenumi}{\ref{intropol}.\arabic{enumi}}
\renewcommand{\labelenumi}{(\thethm.\arabic{enumi})}
\begin{enumerate}
\item \label{poltame} Suppose $p \nmid m$ and $p \nmid e_f(\alpha)$ for each $\alpha \in k$. Let $\Gamma^*$ be 
the set of pairs $(P,Q) \in k \times k$ such that $f(P) = g(Q)$ and $e_f(P) > 1$. Then $\Gamma^*$ is finite, 
say $\Gamma^* = \{ (P_j,Q_j) \colon 1 \le j \le s \}$ with $s = \abs{\Gamma^*}$; 
writing $m_j := e_f(P_j)$ and $n_j := e_g(Q_j)$, we have 
\[
\sum_{i=1}^r (2\g_i-2) = -m - \gcd(m,n) + \sum_{j=1}^s\Bigl(m_j - \gcd(m_j,n_j)\Bigr).
\]
\item \label{polwild} Let $\Lambda^*$ be the set of pairs $(P,Q) \in k \times k$ such that $f(P) = g(Q)$ and 
$e_f(P) > 1$ and $e_g(Q) > 1$. Suppose $\Lambda^* \cup \{(m,n)\}$ does not contain a pair of integers divisible 
by $p$. Then $\Lambda^*$ is finite, say $\Lambda^* = \{ (P_j,Q_j) \colon 1 \le j \le t \}$ with $t = \abs{\Lambda^*}$;
writing $m_j:=e_f(P_j)$ and $n_j:=e_g(Q_j)$, we have
\begin{align*}
\sum_{i=1}^r (2\g_i-2) &= -1 + (m-1)(n-1) - \gcd(m,n)\\
&\qquad - \sum_{j=1}^t\Bigl( (m_j-1)(n_j-1) + \gcd(m_j,n_j)-1 \Bigr).
\end{align*}
\end{enumerate}
\end{cor}
\renewcommand{\theenumi}{\arabic{enumi}}
\renewcommand{\labelenumi}{(\arabic{enumi})}

\begin{rmk}
The formula in \eqref{poltame} is not always valid 
if one only assumes the hypotheses of \eqref{polwild}: for instance, if $f(x) = x^p-x$  and $g(y) = y$ where 
$p>0$ then the left side of the formula in \eqref{poltame} is $-2$ but the right side is $-p-1$. However, the 
formulas in \eqref{rattame} and \eqref{poltame} are useful when they do apply, since 
often there is a point $R\in k$ for which the pairs in $f^{-1}(R)\times g^{-1}(R)$ contribute enough to 
the right sides of these formulas in order to force the left sides to be large. For instance, special cases of these formulas are 
used in this way in \cite{AZ1,AZ0,BT,CHZ,DW,Ritt,Rittrat}.
\end{rmk}

Several authors have proven or stated special cases of these results. For instance:
\begin{itemize}
\item The case $k=\C$, $r=1$, and $\g_1=0$ of \eqref{rattame}  was shown by Ritt \cite[\S 1]{Rittrat}, who had previously treated 
a special case \cite[pp.~60--61]{Ritt}.  Ritt's proof immediately extends to the the case that $r=1$ and $f(x)$ and $g(x)$ are tamely 
ramified \cite[Prop.~2]{Fried}. 
\item The case $\charp(k)=0$ and $r=1$ of \eqref{polwild} is the main result of \cite{Kang};
a further specialization of this case is \cite[Thm.~1.5]{Ng}.
\item The case $k=\mybar{\F_2}$, $f(x) = x^m$, and $g(y) = y^n+y$ of \eqref{poltame} is the main result of \cite{Rod}.
\end{itemize}

Our formulas also yield much simpler proofs of various known results, such as \cite[Thm.~1.1 and 1.3]{Fuji} and \cite[Thm.~1 and 2]{Wang}.  We note that the former reference uses the second main theorem of Nevanlinna theory to prove certain results in case $k=\C$, and our formulas immediately yield stronger results in this case which also extend to arbitrary $k$.

This paper is organized as follows. In the next section we recall some known results which will be used in this paper.
We conclude in Section 3 with proofs of Theorem~\ref{main}, Corollary~\ref{introrat}, and Corollary~\ref{intropol}.

\section{Previous results}

In this section we list the previous results used in this paper. We use the following notation and conventions:
\begin{itemize}
\item $k$ is an algebraically closed field of characteristic $p\ge 0$;
\item all curves are assumed to be smooth, projective, and irreducible;
\item if $C$ is a curve over $k$, then $\g_C$ denotes the genus of $C$, and $C(k)$ denotes the set 
of $k$-rational points on $C$;
\item if $f \colon C\to D$ is a nonconstant morphism between curves over $k$, and $P\in C(k)$, then $e_f(P)$ denotes 
the ramification index of $P$; if in addition $f$ is separable then $d_f(P)$ denotes the different exponent of $P$ 
under $f$.
\end{itemize}

We begin with the Riemann--Hurwitz formula \cite[Thm.~3.4.13]{St}:

\begin{lemma}\label{rh}
If $f \colon C \to D$ is a separable nonconstant morphism of curves over $k$, then
\[
2\g_C-2 = \deg(f) \cdot (2\g_D-2) + \sum_{P\in C(k)} d_f(P).
\]
\end{lemma}

The ramification index \cite[Def.~3.1.5]{St} and different exponent \cite[Def.~3.4.3]{St}
are related by Dedekind's different theorem, e.g.\ cf.\ \cite[Prop.~III.13]{Se} or \cite[Thm.~3.5.1]{St}:

\begin{lemma}\label{d}
If $f \colon C \to D$ is a separable nonconstant morphism of curves over $k$, then for any $P \in C(k)$ we have 
$d_f(P) \ge e_f(P)-1$, with equality holding if and only if $P$ is tamely ramified under $f$.  
\end{lemma}

The next result describes the ramification index and different exponent in a tower of curves 
\cite[Prop.~3.1.6 and Cor.~3.4.12]{St}:

\begin{lemma} \label{tower}
Let $f \colon C \to D$ and $g \colon D \to E$ be nonconstant morphisms of curves over $k$. Then for any $P \in C(k)$ we have 
\[
e_{g \circ f}(P) = e_f(P) \cdot e_g(f(P)).
\] 
If in addition both $f$ and $g$ are separable then
\[
d_{g \circ f}(P) = e_f(P) \cdot d_g(f(P)) + d_f(P).
\]
\end{lemma}

We will use the following Fundamental Equality \cite[Thm.~3.1.11]{St}: 

\begin{lemma}\label{fund}
If $f \colon C \to D$ is a nonconstant morphism of curves over $k$, then for any $Q \in D(k)$ we have 
\[
\deg(f)=\sum_{P \in f^{-1}(Q)} e_f(P).
\]
\end{lemma}

We will also use the following version of Abhyankar's lemma \cite[Thm.~3.9.1 and Prop.~3.10.2]{St}:

\begin{lemma}\label{abh}
Let $f\colon A\to D$ and $g\colon B\to D$ be nonconstant morphisms of curves over $k$, and let $C$ be the normalization 
of a component of the fiber product of $f$ and $g$, with induced maps $\phi \colon C \to A$ 
and $\psi \colon C \to B$. For any $S \in C(k)$ such that $\phi(S)$ is tamely ramified under $f$, we have 
\[
e_{f \circ \phi}(S) = \lcm\bigl(e_f(\phi(S)),e_g(\psi(S))\bigr).
\]
\end{lemma}

The next result was proved in \cite{DZ}:

\begin{lemma}\label{df}
Assume $f \colon A \to D$ and $g \colon B \to D$ are nonconstant morphisms of curves over $k$. Let $C_1,\dots,C_r$ 
be the normalizations of the irreducible components of the fiber product of $f$ and $g$, with induced morphisms 
$\phi_i \colon C_i \to A$ and $\psi_i \colon C_i \to B$. Then for any pair $(P,Q) \in A(k) \times B(k)$ such that
$f(P) = g(Q)$ and $p \nmid \gcd\bigl(e_f(P),e_g(Q)\bigr)$ we have
\[
\sum_{i=1}^r \abs{\phi_i^{-1}(P)\cap\psi_i^{-1}(Q)}=\gcd\bigl(e_f(P),e_g(Q)\bigr).
\]
\end{lemma}


\section{Proofs of main results}

Corollaries~\ref{introrat} and \ref{intropol} follow immediately from Theorem~\ref{main}, so we need only prove the latter 
result. Throughout this section we assume the hypotheses and notation from Theorem~\ref{main}. For $i\in\{1,2,\dots,r\}$, 
let $\phi_i\colon C_i\to A$ and $\psi_i\colon C_i\to B$ be the maps arising from the fibered product. We begin by combining 
the Riemann--Hurwitz formula for the various maps $\psi_i$ into a single formula. We express this combined formula in terms 
of a summation whose summands will turn out to be convenient for explicit computation in the setting of Theorem~\ref{main}.

\begin{lemma} \label{verywild} 
If $f$ is separable then
\begin{align}
\label{321}\sum_{i=1}^r (2\g_i-2) &= m (2\g_B-2) +  \sum_{i=1}^r \sum_{S \in C_i(k)} d_{\psi_i}(S) \\
\begin{split}
&= \label{322}m (2\g_B-2) + n (2g_A-2) - mn (2\g_D-2) \\
 &\qquad- \sum_{i=1}^r \sum_{S\in C_i(k)}\Bigl(e_{\phi_i}(S) \cdot d_f(\phi_i(S)) - d_{\psi_i}(S) \Bigr).
\end{split}
\end{align}
If both $f$ and $g$ are separable then each $S\in C_i(k)$ satisfies
\begin{equation}
\label{diffexp}
e_{\phi_i}(S) \cdot d_f(\phi_i(S)) - d_{\psi_i}(S) = d_{f \circ \phi_i}(S) - d_{\phi_i}(S) - d_{\psi_i}(S),
\end{equation}
which yields the symmetric equation
\begin{equation}
\begin{aligned}
\label{323}\sum_{i=1}^r (2\g_i-2) &=  m (2\g_B-2) + n (2g_A-2) - mn (2\g_D-2) \\
 &\qquad- \sum_{i=1}^r 
\sum_{S \in C_i(k)} \Bigl( d_{f \circ \phi_i}(S) - d_{\phi_i}(S) - d_{\psi_i}(S)   \Bigr).
\end{aligned}
\end{equation}
\end{lemma} 

\begin{proof}
Since $f$ is separable, also each $\psi_i$ is separable. Riemann--Hurwitz for $\psi_i$ says that
\[ 
2\g_i-2 = \deg(\psi_i) \cdot (2\g_B-2) + \sum_{S \in C_i(k)} d_{\psi_i}(S).
\]
By summing over all $i$, and using the fact that $\sum_{i=1}^r \deg(\psi_i) = m$, this yields \eqref{321}.
Riemann--Hurwitz for the map $f \colon A \to D$ says that
\begin{equation} \label{use1}
2\g_A-2 = m (2\g_D-2) + \sum_{P \in A(k)} d_f(P).
\end{equation}
For each $P \in A(k)$, Lemma~\ref{fund} yields
\begin{equation}\label{use2}
\sum_{i=1}^r \sum_{S \in \phi_i^{-1}(P)} e_{\phi_i}(S) = \sum_{i=1}^r \deg(\phi_i) = n.
\end{equation}
Multiply \eqref{use1} by $n$, and then substitute \eqref{use2}, to get
\begin{align*}
n(2\g_A-2) - mn(2\g_D-2) &= \sum_{P \in A(k)} n d_f(P) \\
&= \sum_{P \in A(k)} \sum_{i=1}^r \sum_{S \in \phi_i^{-1}(P)} e_{\phi_i}(S) \cdot d_f(P) \\
&= \sum_{i=1}^r \sum_{S\in C_i(k)} e_{\phi_i}(S) \cdot d_f(\phi_i(S)).
\end{align*}
Upon subtracting the right side of this equation from the left, and adding the difference to the right side of \eqref{321}, we obtain \eqref{322}.

If both $f$ and $g$ are separable then also each $\phi_i$ is separable, so by Lemma~\ref{tower} we have
\[
e_{\phi_i}(S)\cdot d_f(\phi_i(S)) = d_{f\circ\phi_i}(S)-d_{\phi_i}(S)
\]
for each $S\in C_i(k)$.  Subtracting $d_{\psi_i}(S)$ from both sides yields \eqref{diffexp}, and combining \eqref{322} with \eqref{diffexp} yields \eqref{323}.
\end{proof}

In light of Lemma~\ref{verywild}, in order to prove Theorem~\ref{main} we must compute the summands involving different exponents in the right sides of \eqref{321}, \eqref{322} and \eqref{323}.  We do this in the next result, under suitable hypotheses about tame ramification.

\begin{lemma} \label{local}
Pick $S\in C_i(k)$ for some $i$, and put $P:=\phi_i(S)$ and $Q:=\psi_i(S)$.  If $P$ is tamely ramified under $f$ then
\begin{equation}\label{371}
d_{\psi_i}(S) = \frac{e_f(P)}{\gcd\bigl(e_f(P),e_g(Q)\bigr)} - 1
\end{equation}
and
\begin{equation}\label{372}
\begin{aligned}
&\gcd\bigl(e_f(P),e_g(Q)\bigr) \cdot \Bigl( e_{\phi_i}(S) \cdot d_f(P) - d_{\psi_i}(S)  \Bigr) \\
&\quad =  (e_f(P)-1) \cdot (e_g(Q)-1) + \gcd\bigl(e_f(P),e_g(Q)\bigr) - 1.
\end{aligned}
\end{equation}
%
%
\end{lemma}

\begin{proof}
By Lemma~\ref{abh} we have
\[
e_{f \circ \phi_i}(S) = \lcm\bigl(e_f(P),e_g(Q)\bigr) = \frac{e_f(P) \cdot e_g(Q)}{\gcd\bigl(e_f(P),e_g(Q)\bigr)}.
\]
Lemma~\ref{tower} yields
\[
e_{\phi_i}(S) \cdot e_f(P) = e_{f \circ \phi_i}(S) =  e_{g \circ \psi_i}(S) = e_{\psi_i}(S) \cdot e_g(Q), 
\]
so that
\[
e_{\phi_i}(S) = \frac{e_g(Q)}{\gcd\bigl(e_f(P),e_g(Q)\bigr)} \quad\text{ and }\quad e_{\psi_i}(S) = \frac{e_f(P)}{\gcd\bigl(e_f(P),e_g(Q)\bigr)}.
\]
In particular, $S$ is tamely ramified under $\psi_i$, so Lemma~\ref{d} implies that
\[
d_{\psi_i}(S) = e_{\psi_i}(S) - 1 = \frac{e_f(P)}{\gcd\bigl(e_f(P),e_g(Q)\bigr)} - 1,
\]
which is \eqref{371}.  Lemma~\ref{d} also gives $d_f(P)=e_f(P)-1$, so that
\begin{align*}
&\gcd\bigl(e_f(P),e_g(Q)\bigr)\cdot\Bigl(e_{\phi_i}(S)\cdot d_f(P)-d_{\psi_i}(S)\Bigr) \\
&\quad = e_g(Q)\cdot(e_f(P)-1) - \Bigl(e_f(P)-\gcd\bigl(e_f(P),e_g(Q)\bigr)\Bigr) \\
&\quad = (e_g(Q)-1)\cdot (e_f(P)-1) + \gcd\bigl(e_f(P),e_g(Q)\bigr) - 1,
\end{align*}
which is \eqref{372}.
\end{proof}

We now prove Theorem~\ref{main}.

\begin{proof}[Proof of Theorem~\ref{main}]
If $f$ is tamely ramified then in particular $f$ is separable, so that \eqref{321} yields
\[
\sum_{i=1}^r (2\g_i-2) = m (2\g_B-2) +  \sum_{i=1}^r \sum_{S \in C_i(k)} d_{\psi_i}(S).
\]
By \eqref{371} the right side equals
\begin{align*}
&m (2\g_B-2) +  \sum_{i=1}^r \sum_{S \in C_i(k)} \Bigl(\frac{e_f(\phi_i(S))}{\gcd\bigl(e_f(\phi_i(S)),e_g(\psi_i(S))\bigr)} - 1\Bigr)\\
&\quad = m (2\g_B-2) + \sum_{R\in D(k)}\sum_{\substack{P\in f^{-1}(R) \\ Q\in g^{-1}(R)}} \Bigl(\frac{e_f(P)}{\gcd\bigl(e_f(P),e_g(Q)\bigr)}-1\Bigr)\cdot N_{P,Q}
\end{align*}
where
\[
N_{P,Q}:=\sum_{i=1}^r\abs{\phi_i^{-1}(P)\cap\psi_i^{-1}(Q)}
\]
for all $P\in A(k)$ and $Q\in B(k)$ satisfying $f(P)=g(Q)$.
Substituting the value $N_{P,Q}=\gcd\bigl(e_f(P),e_g(Q)\bigr)$ from Lemma~\ref{df} yields \eqref{maintame}. 

Likewise, if $g$ is tamely ramified then the combination of \eqref{322}, \eqref{372} and Lemma~\ref{df} yields the conclusion of \eqref{mainwild}.  Since the conclusion of \eqref{mainwild} is symmetric in $f$ and $g$, it follows that this conclusion is also true if $g$ is tamely ramified.

It remains to prove \eqref{mainwild} when neither $f$ nor $g$ is tamely ramified.
Thus $f$ has a wildly ramified branch point, which by the hypothesis of \eqref{mainwild} is not a wildly ramified branch point of $g$, so that $g$ is separable.  Likewise $f$ is separable.  For any $i$ and any $S\in C_i(k)$, put $P:=\phi_i(S)$ and $Q:=\psi_i(S)$.  Since $f(P)=g(Q)$ but $f$ and $g$ have no common wildly ramified branch points, at least one of $e_f(P)$ and $e_g(Q)$ is not divisible by $p$.  If $p\nmid e_f(P)$ then \eqref{diffexp} and \eqref{372} yield
\begin{align*}
&d_{f \circ \phi_i}(S)-d_{\phi_i}(S)-d_{\psi_i}(S)
= e_{\phi_i}(S)\cdot d_f(P)-d_{\psi_i}(S) \\
&\qquad
=\frac{(e_f(P)-1)\cdot(e_g(Q)-1)+\gcd\bigl(e_f(P),e_g(Q)\bigr)-1}{\gcd\bigl(e_f(P),e_g(Q)\bigr)}.
\end{align*}
If $p\nmid e_g(Q)$ then the same argument yields the same conclusion.  Thus the above equation is true in any case.  Combining it with 
\eqref{323} and Lemma~\ref{df} yields \eqref{mainwild}.
\end{proof}



\end{document}